\documentclass[12pt]{article}

\usepackage[T1]{fontenc}
\usepackage{avant}
\usepackage[cp1250]{inputenc}
\usepackage{amsmath,amssymb,amsthm}
\usepackage{geometry}
\usepackage{caption}
\usepackage{float}
\usepackage{graphicx}

\title{A restriction on proper actions on homogeneous spaces of a reductive type}

\author{Maciej Boche\'nski and Marek Ogryzek}

\begin{document}

\newtheorem{theorem}{Theorem}
\newtheorem{proposition}{Proposition}
\newtheorem{lemma}{Lemma}
\newtheorem{definition}{Definition}
\newtheorem{example}{Example}
\newtheorem{note}{Note}
\newtheorem{corollary}{Corollary}
\newtheorem{remark}{Remark}

\maketitle{}

\begin{abstract}
Let $L$ be a reductive subgroup of a reductive Lie group $G$. Let $G/H$ be a homogeneous space of reductive type. We provide a necessary condition for the properness of the action of $L$ on $G/H$. As an application we give examples of spaces that do not admit standard compact Clifford-Klein forms.
\end{abstract}

\section{Introduction}

Let $L$ be a locally compact topological group acting continuously on a locally Hausdorff topological space $M$. This action is called {\it \textbf{proper}} if for every compact subset $C \subset M$ the set
$$L(C):=\{  g\in L \ | \ g\cdot C \cap C \neq \emptyset \}$$
is compact. In this paper, our main concern is the following question posed by T. Kobayashi \cite{kob3}
\begin{center}
How ``large'' subgroups of $G$ can act properly 
\end{center}
\begin{center}
on a homogeneous space $G/H$?  \textbf{(Q1)}
\end{center}

We restrict our attention to the case where $M=G/H$ is a homogeneous space of reductive type and always assume that $G$ is a linear connected reductive real Lie group with the Lie algebra $\mathfrak{g}.$ Let $H\subset G$ be a closed subgroup of $G$ with finitely many connected components and $\mathfrak{h}$ be the Lie algebra of $H.$
\begin{definition}
 The subgroup $H$ is reductive in $G$ if $\mathfrak{h}$ is reductive in $\mathfrak{g},$ that is, there exists a Cartan involution $\theta $ for which $\theta (\mathfrak{h}) = \mathfrak{h}.$
 The space $G/H$ is called the homogeneous space of reductive type.
 \label{def1}
\end{definition}

\noindent
Note that if $\mathfrak{h}$ is reductive in $\mathfrak{g}$ then $\mathfrak{h}$ is a reductive Lie algebra. 

It is natural to ask when does a closed subgroup of $G$ act properly on a space of reductive type $G/H.$ This problem was treated, inter alia, in \cite{ben}, \cite{bt}, \cite{kas}, \cite{kob2}, \cite{kob4}, \cite{kob1}, \cite{kul} and \cite{ok}. In \cite{kob2} one can find a very important criterion for a proper action of a subgroup $L$ reductive in $G.$ To state this criterion we need to introduce some additional notation. Let $\mathfrak{l}$ be the Lie algebra of $L.$ Take a Cartan involution $\theta$ of $\mathfrak{g}.$ We obtain the Cartan decomposition
\begin{equation}
 \mathfrak{g}=\mathfrak{k} + \mathfrak{p}.
 \label{eq1}
\end{equation}
Choose a maximal abelian subspace $\mathfrak{a}$ in $\mathfrak{p}.$ The subspace $\mathfrak{a}$ is called the \textbf{\textit{maximally split abelian subspace}} of $\mathfrak{p}$ and $\text{rank}_{\mathbb{R}}(\mathfrak{g}) := \text{dim} (\mathfrak{a})$ is called the \textbf{\textit{real rank}} of $\mathfrak{g}.$ It follows from Definition \ref{def1} that $\mathfrak{h}$ and $\mathfrak{l}$ admit Cartan decompositions 
$$\mathfrak{h}=\mathfrak{k}_{1} + \mathfrak{p}_{1} \ \text{and} \ \mathfrak{l}=\mathfrak{k}_{2} + \mathfrak{p}_{2},$$
given by Cartan involutions $\theta_{1}, \ \theta_{2}$ of $\mathfrak{g}$ such that $\theta_{1} (\mathfrak{h})= \mathfrak{h}$ and $\theta_{2} (\mathfrak{l})= \mathfrak{l}.$ Let $\mathfrak{a}_{1} \subset \mathfrak{p}_{1}$ and $\mathfrak{a}_{2} \subset \mathfrak{p}_{2}$ be maximally split abelian subspaces of $\mathfrak{p}_{1}$ and $\mathfrak{p}_{2},$ respectively. One can show that there exist $a,b \in G$ such that $\mathfrak{a}_{\mathfrak{h}} := \text{\rm Ad}_{a}\mathfrak{a}_{1} \subset \mathfrak{a}$ and $\mathfrak{a}_{\mathfrak{l}} := \text{\rm Ad}_{b}\mathfrak{a}_{2} \subset \mathfrak{a}.$ Denote by $W_{\mathfrak{g}}$ the Weyl group of $\mathfrak{g}.$ In this setting the following holds
\begin{theorem}[Theorem 4.1 in \cite{kob2}] The following three conditions are equivalent
 \begin{enumerate}
   \item $L$ acts on $G/H$ properly.
   \item $H$ acts on $G/L$ properly.
   \item For any $w \in W_{\mathfrak{g}},$ $w\cdot \mathfrak{a}_{\mathfrak{l}} \cap \mathfrak{a}_{\mathfrak{h}} =\{ 0 \}.$ 
 \end{enumerate}
 \label{twkob}
\end{theorem}

\noindent
Note that the criterion 3. in Theorem \ref{twkob} depends on how $L$ and $H$ are embedded in $G$ up to inner-automorphisms. Theorem \ref{twkob} provides a partial answer to Q1.
\begin{corollary}[Corollary 4.2 in \cite{kob2}] If $L$ acts properly on $G/H$ then
$$\text{\rm rank}_{\mathbb{R}}(\mathfrak{l}) + \text{\rm rank}_{\mathbb{R}}(\mathfrak{h}) \leq \text{\rm rank}_{\mathbb{R}} (\mathfrak{g}).$$
\label{coko}
\end{corollary}

\noindent
Hence the real rank of $L$ is bounded by a constant which depends on $G/H,$ no matter how $H$ and $L$ are embedded in $G.$ In this paper we find a similar, stronger restriction for Lie groups $G,H,L$ by means of a certain tool which we call the a-hyperbolic rank (see Section 2, Definition \ref{dd2} and Table \ref{tab1}). In more detail we prove the following
\begin{theorem}
If $L$ acts properly on $G/H$ then
$$\mathop{\mathrm{rank}}\nolimits_{\text{\rm a-hyp}}(\mathfrak{l}) + \mathop{\mathrm{rank}}\nolimits_{\text{\rm a-hyp}}(\mathfrak{h}) \leq \mathop{\mathrm{rank}}\nolimits_{\text{\rm a-hyp}} (\mathfrak{g}).$$
\label{twgl}
\end{theorem}

Recall that a homogeneous space $G/H$ of reductive type admits a \textbf{\textit{compact Clifford-Klein form}} if there exists a discrete subgroup $\Gamma \subset G$ such that $\Gamma$ acts properly on $G/H$ and $\Gamma \backslash G/H$ is compact. The space $G/H$ admits a \textbf{\textit{standard compact Clifford-Klein form}} in the sense of Kassel-Kobayashi \cite{kako} if there exists a subgroup $L$ reductive in $G$ such that $L$ acts properly on $G/H$ and $L \backslash G/H$ is compact. In the latter case, for any discrete cocompact subgroup $\Gamma ' \subset L,$ the space $\Gamma ' \backslash G/H$ is a compact Clifford-Klein form. Therefore it follows from Borel's theorem (see \cite{bor}) that any homogeneous space of reductive type admitting a standard compact Clifford-Klein form also admits a compact Clifford-Klein form. 
\newline
It is not known if the converse statement holds, but all known reductive homogeneous spaces $G/H$ admitting compact Clifford-Klein forms also admit standard compact Clifford-Klein forms.

As a corollary to Theorem \ref{twgl}, we get examples of the semisimple symmetric spaces without standard compact Clifford-Klein forms. In particular, we cannot find the first example in the existing literature:

\begin{corollary}
The homogeneous spaces $G/H=SL(2k+1, \mathbb{R})/SO(k-1,k+2)$ \linebreak and $G/H=SL(2k+1, \mathbb{R})/Sp(k-1,\mathbb{R})$ for $k \geq 5$ do not admit standard compact Clifford-Klein forms.
\label{co1}
\end{corollary}

\begin{remark}
Let us mention the following results, related to the above corollary.
\begin{itemize}
	\item T. Kobayashi proved in \cite{kobadm} that $SL(2k,\mathbb{R})/SO(k,k)$ for $k\geq 1$ and $SL(n,\mathbb{R})/Sp(l,\mathbb{R})$ \linebreak for $0<2l \leq n-2$ do not admit compact Clifford-Klein forms.
	\item Y. Benoist proved in \cite{ben} that $SL(2k+1,\mathbb{R})/SO(k,k+1)$ for $k\geq 1$ does not admit compact Clifford-Klein forms.
	\item Y. Morita proved recently in \cite{mor} that $SL(p+q,\mathbb{R})/SO(p,q)$ does not admit compact Clifford-Klein forms if $p$ and $q$ are both odd.
\end{itemize}
Note that these works are devoted to the problem of existence of compact Clifford-Klein forms on a given homogeneous space (not only standard compact Clifford-Klein forms).
\end{remark}

\section{The a-hyperbolic rank and antipodal hyperbolic orbits}

Let $\Sigma_{\mathfrak{g}}$ be a system of restricted roots for $\mathfrak{g}$ with respect to $\mathfrak{a}.$ Choose a system of positive roots $\Sigma^{+}_{\mathfrak{g}}$ for $\Sigma_{\mathfrak{g}}.$ Then a fundamental domain of the action of $W_{\mathfrak{g}}$ on $\mathfrak{a}$ can be define as
$$\mathfrak{a}^{+} := \{  X\in \mathfrak{a} \ |  \ \alpha (X) \geq 0 \ \text{\rm for any} \ \alpha \in \Sigma^{+}_{\mathfrak{g}} \}.$$
Note that 
$$sX+tY \in \mathfrak{a}^{+},$$
for any $s,t \geq 0$ and $X,Y\in \mathfrak{a}^{+}$. Therefore $\mathfrak{a}^{+}$ is a convex cone in the linear space $\mathfrak{a}.$
Let $w_{0} \in W_{\mathfrak{g}}$ be the longest element. One can show that
$$-w_{0}: \mathfrak{a} \rightarrow \mathfrak{a}, \ \ X \mapsto -(w_{0} \cdot X)$$
is an involutive automorphism of $\mathfrak{a}$ preserving $\mathfrak{a}^{+}.$ Let $\mathfrak{b} \subset \mathfrak{a}$ be the subspace of all fixed points of $-w_{0}$ and put
$$\mathfrak{b}^{+} := \mathfrak{b} \cap \mathfrak{a}^{+}.$$
Thus $\mathfrak{b}^{+}$ is a convex cone in $\mathfrak{a}.$ We also have $\mathfrak{b} = \text{Span} (\mathfrak{b}^{+}).$
\begin{definition}
The dimension of $\mathfrak{b}$ is called the a-hyperbolic rank of $\mathfrak{g}$ and is denoted by 
$$\mathop{\mathrm{rank}}\nolimits_{\text{a-hyp}} (\mathfrak{g}).$$
\label{dd2}
\end{definition}

\noindent
The a-hyperbolic ranks of the simple real Lie algebras can be deduce from Table \ref{tab1}.
\begin{center}
 \begin{table}[h]
 \centering
 {\footnotesize
 \begin{tabular}{| c | c | c |}
   \hline
   \multicolumn{3}{|c|}{ \textbf{\textit{A-HYPERBOLIC RANK}}} \\
   \hline                        
   $\mathfrak{g}$ & $\mathop{\mathrm{rank}}\nolimits_{\text{a-hyp}} (\mathfrak{g})$ & $\text{rank}_{\mathbb{R}} (\mathfrak{g})$ \\
   \hline
   $\mathfrak{sl}(2k,\mathbb{R})$  & $k$ &  $2k-1$ \\
   {\scriptsize $k\geq 2$} & & \\
   \hline
   $\mathfrak{sl}(2k+1,\mathbb{R})$  & $k$ & $2k$ \\
   {\scriptsize $k\geq 1$} & & \\
   \hline
   $\mathfrak{su}^{\ast}(4k)$  & $k$ & $2k-1$ \\
   {\scriptsize $k\geq 2$} & & \\
   \hline
   $\mathfrak{su}^{\ast}(4k+2)$  & $k$ & $2k$ \\
   {\scriptsize $k\geq 1$} & & \\
   \hline
   $\mathfrak{so}(2k+1,2k+1)$  & $2k$ & $2k+1$ \\
   {\scriptsize $k\geq 2$} & &  \\
   \hline
	 $\mathfrak{e}_{6}^{\text{I}}$ & 4 & 6 \\
	 \hline
   $\mathfrak{e}_{6}^{\text{IV}}$  & 1 & 2 \\
   \hline  
 \end{tabular}
 }

 \caption{
 This table contains all real forms of simple Lie algebras $\mathfrak{g}^{\mathbb{C}}$ for which $\text{rank}_{\mathbb{R}}(\mathfrak{g}) \neq \mathop{\mathrm{rank}}\nolimits_{\text{a-hyp}}(\mathfrak{g}).$ The notation is close to \cite{ov2}, Table 9, pages 312-316.
 }
 \label{tab1}
 \end{table}
\end{center}

\noindent
A method of calculation of the a-hyperbolic rank of a simple Lie algebra can be found in \cite{bt}. The a-hyperbolic rank of a semisimple Lie algebra equals the sum of a-hyperbolic ranks of all its simple parts. For a reductive Lie algebra $\mathfrak{g}$ we put
$$\mathop{\mathrm{rank}}\nolimits_{\text{a-hyp}} (\mathfrak{g}) := \mathop{\mathrm{rank}}\nolimits_{\text{a-hyp}}([\mathfrak{g},\mathfrak{g}]).$$

There is a close relation between $\mathfrak{b}^{+}$ and the set of antipodal hyperbolic orbits in $\mathfrak{g}.$ We say that an element $X \in \mathfrak{g}$ is {\it \textbf{hyperbolic}}, if $X$ is semisimple (that is, $\mathrm{ad}_{X}$ is diagonalizable) and all eigenvalues of $\mathrm{ad}_{X}$ are real.
\begin{definition}
An adjoint orbit $O_{X}:=\mathop{\mathrm{Ad}}\nolimits (G)(X)$ is said to be hyperbolic if $X$ (and therefore every element of $O_{X}$) is hyperbolic. An adjoint orbit $O_{Y}$ is antipodal if $-Y\in O_{Y}$ (and therefore for \linebreak every $Z\in O_{Y},$ $-Z\in O_{Y}$).
\end{definition} 

\begin{lemma}[c.f. Fact 5.1 and Lemma 5.3 in \cite{ok}]
There is a bijective correspondence between antipodal hyperbolic orbits $O_{X}$ in $\mathfrak{g}$ and elements $Y \in \mathfrak{b}^{+}.$ This correspondence is given by
$$\mathfrak{b}^{+}\ni Y \mapsto O_{Y}.$$
Furthermore, for every hyperbolic orbit $O_{X}$ in $\mathfrak{g}$ the set $O_{X} \cap \mathfrak{a}$ is a single $W_{\mathfrak{g}}$ orbit in $\mathfrak{a}$.
\label{lma}
\end{lemma}

\section{The main result}

We need two basic facts from linear algebra.

\begin{lemma}
Let $V_{1},V_{2}$ be vector subspaces of a real linear space $V$ of finite dimension. Then
$$\text{\rm dim} (V_{1}+V_{2})= \text{\rm dim} (V_{1}) + \text{\rm dim} (V_{2}) - \text{\rm dim} (V_{1}\cap V_{2}).$$
\label{lma1}
\end{lemma}

\begin{lemma}
Let $V_{1},...,V_{n}$ be a collection of vector subspaces of a real linear space $V$ of a finite dimension and let $A^{+} \subset V$ be a convex cone. Assume that
$$A^{+} \subset \bigcup_{k=1}^{n} V_{k}.$$
Then there exists a number $k,$ such that $A^{+} \subset V_{k}.$ 
\label{lma2}
\end{lemma}

We also need the following, technical lemma. Choose a subalgebra $\mathfrak{h}$ reductive in $\mathfrak{g}$ which corresponds to a Lie group $H \subset G.$ Let $\mathfrak{b}_{[\mathfrak{h},\mathfrak{h}]}^{+} \subset \mathfrak{a}_{\mathfrak{h}}$ be the convex cone constructed according to the procedure described in the previous subsection (for  $[\mathfrak{h},\mathfrak{h}]$).

\begin{lemma}
Let $X\in \mathfrak{b}_{[\mathfrak{h},\mathfrak{h}]}^{+}.$ The orbit $O_{X}:=\mathop{\mathrm{Ad}}\nolimits (G)(X)$ is an antipodal hyperbolic orbit in $\mathfrak{g}.$
\label{lma3}
\end{lemma}

\begin{proof}
By Lemma \ref{lma} the vector $X$ defines an antipodal hyperbolic orbit in $\mathfrak{h}.$ Therefore we can find $h \in H \subset G$ such that
$\text{\rm Ad}_{h}(X) = - X$. Since a maximally split abelian subspace $\mathfrak{a} \subset \mathfrak{g}$ consists of vectors for which $\mathrm{ad}$ is diagonalizable with real values and
$$X \in \mathfrak{b}_{[\mathfrak{h},\mathfrak{h}]}^{+} \subset \mathfrak{a}_{\mathfrak{h}} \subset \mathfrak{a},$$
thus the vector $X$ is hyperbolic in $\mathfrak{g}.$ It follows that $\mathop{\mathrm{Ad}}\nolimits (G)(X)$ is a hyperbolic orbit in $\mathfrak{g}$ \linebreak and $-X \in \mathop{\mathrm{Ad}}\nolimits (G)(X).$
\end{proof}
Now we are ready to give a proof of Theorem \ref{twgl}.

\begin{proof}
Assume that $\mathop{\mathrm{rank}}\nolimits_{\text{a-hyp}}(\mathfrak{l}) + \mathop{\mathrm{rank}}\nolimits_{\text{a-hyp}}(\mathfrak{h}) > \mathop{\mathrm{rank}}\nolimits_{\text{a-hyp}} (\mathfrak{g})$ and let $\mathfrak{b}_{[\mathfrak{h},\mathfrak{h}]}^{+},$ $\mathfrak{b}_{[\mathfrak{l},\mathfrak{l}]}^{+},$ $\mathfrak{b}^{+}$ be appropriate convex cones. If $X \in \mathfrak{b}_{[\mathfrak{h},\mathfrak{h}]}^{+}$ then $O_{X}^{H} := \mathop{\mathrm{Ad}}\nolimits (G)(X)$ is an antipodal hyperbolic orbit in $\mathfrak{h}.$ By Lemma \ref{lma3} the orbit $O_{X}^{G} := \mathop{\mathrm{Ad}}\nolimits (G)(X)$ is an antipodal hyperbolic orbit in $\mathfrak{g}.$ By Lemma \ref{lma} there exists $Y \in \mathfrak{b}^{+}$ such that
$$O_{X}^{G} = O_{Y}^{G} = \mathop{\mathrm{Ad}}\nolimits (G)(Y).$$
Since $\mathfrak{b}_{[\mathfrak{h},\mathfrak{h}]}^{+} \subset \mathfrak{a}_{\mathfrak{h}} \subset \mathfrak{a}$ and $\mathfrak{b}^{+} \subset \mathfrak{a}$ thus (according to Lemma \ref{lma}) we get $X=w_{1} \cdot Y$ for a certain $w_{1} \in W_{\mathfrak{g}}.$ Therefore
$$\mathfrak{b}_{[\mathfrak{h},\mathfrak{h}]}^{+} \subset W_{\mathfrak{g}} \cdot \mathfrak{b}^{+} = \bigcup_{w\in W_{\mathfrak{g}}} w \cdot \mathfrak{b}^{+} \subset \bigcup_{w\in W_{\mathfrak{g}}} w \cdot \mathfrak{b}.$$
Analogously
$$\mathfrak{b}_{[\mathfrak{l},\mathfrak{l}]}^{+} \subset \bigcup_{w\in W_{\mathfrak{g}}} w \cdot \mathfrak{b}^{+} \subset \bigcup_{w\in W_{\mathfrak{g}}} w \cdot \mathfrak{b}.$$
By Lemma \ref{lma2} there exist $w_{\mathfrak{h}},w_{\mathfrak{l}} \in W_{\mathfrak{g}}$ such that
$$\mathfrak{b}_{[\mathfrak{h},\mathfrak{h}]}^{+} \subset w_{\mathfrak{h}}^{-1} \cdot \mathfrak{b} \ \ \text{\rm and} \ \ \mathfrak{b}_{[\mathfrak{l},\mathfrak{l}]}^{+} \subset w_{\mathfrak{l}}^{-1} \cdot \mathfrak{b},$$
because $W_{\mathfrak{g}}$ acts on $\mathfrak{a}$ by linear transformations. Therefore
$$\mathfrak{b}_{[\mathfrak{h},\mathfrak{h}]} \subset w_{\mathfrak{h}}^{-1} \cdot \mathfrak{b} \ \ \text{\rm and} \ \ \mathfrak{b}_{[\mathfrak{l},\mathfrak{l}]} \subset w_{\mathfrak{l}}^{-1} \cdot \mathfrak{b}$$
where $\mathfrak{b}_{[\mathfrak{h},\mathfrak{h}]} := \text{Span}(\mathfrak{b}_{[\mathfrak{h},\mathfrak{h}]}^{+})$ and $\mathfrak{b}_{[\mathfrak{l},\mathfrak{l}]} := \text{Span} (\mathfrak{b}_{[\mathfrak{l},\mathfrak{l}]}^{+}).$ We obtain
$$w_{\mathfrak{h}} \cdot \mathfrak{b}_{[\mathfrak{h},\mathfrak{h}]} \subset \mathfrak{b} , \ w_{\mathfrak{l}} \cdot \mathfrak{b}_{[\mathfrak{l},\mathfrak{l}]} \subset \mathfrak{b}.$$
By the assumption and Lemma \ref{lma1}
$$\text{\rm dim}(w_{\mathfrak{h}} \cdot \mathfrak{b}_{[\mathfrak{h},\mathfrak{h}]} \cap w_{\mathfrak{l}} \cdot \mathfrak{b}_{[\mathfrak{l},\mathfrak{l}]}) >0.$$
Choose $0 \neq Y \in w_{\mathfrak{h}} \cdot \mathfrak{b}_{[\mathfrak{h},\mathfrak{h}]} \cap w_{\mathfrak{l}} \cdot \mathfrak{b}_{[\mathfrak{l},\mathfrak{l}]}.$ Then
$$w_{\mathfrak{l}} \cdot X_{\mathfrak{l}}=Y= w_{\mathfrak{h}} \cdot X_{\mathfrak{h}} \ \text{\rm for some} \ X_{\mathfrak{h}} \in \mathfrak{b}_{[\mathfrak{h},\mathfrak{h}]}\backslash \{ 0 \}  \ \text{\rm and} \ X_{\mathfrak{l}} \in \mathfrak{b}_{[\mathfrak{l},\mathfrak{l}]}\backslash \{  0 \}. $$
Take $w_{2} := w_{\mathfrak{h}}^{-1}w_{\mathfrak{l}} \in W_{\mathfrak{g}},$ we have $X_{\mathfrak{h}}= w_{2} \cdot X_{\mathfrak{l}}$ and $X_{\mathfrak{h}} \in \mathfrak{a}_{\mathfrak{h}}, \ X_{\mathfrak{l}} \in \mathfrak{a}_{\mathfrak{l}}.$ Thus $0 \neq X_{\mathfrak{h}} \in w_{2} \cdot \mathfrak{a}_{\mathfrak{l}} \cap \mathfrak{a}_{\mathfrak{h}}.$ The assertion follows from Theorem \ref{twkob}. 
\end{proof}

We can proceed to a proof of Corollary \ref{co1}. For a reductive Lie group $D$ with a Lie algebra $\mathfrak{d}$ with a Cartan decomposition
$$\mathfrak{d} = \mathfrak{k}_{\mathfrak{d}} + \mathfrak{p}_{\mathfrak{d}}$$
we define $d(G) := \text{dim} (\mathfrak{p}_{\mathfrak{d}}).$ We will need the following properties
\begin{theorem}[Theorem 4.7 in \cite{kob2}]
Let $L$ be a subgroup reductive in $G$ acting properly on $G/H.$ The space $L \backslash G /H$ is compact if and only if 
$$d(L)+d(H)=d(G).$$
\label{twkk}
\end{theorem}
\begin{theorem}[\cite{yo}] 
If $J \subset G$ is a semisimple subgroup then it is reductive in $G.$
\label{twy}
\end{theorem}

\begin{proposition}
Let $L \subset G$ be a semisimple Lie group acting properly on  \linebreak $G/H=SL(2k+1, \mathbb{R})/SO(k-1,k+2)$ or $G/H=SL(2k+1, \mathbb{R})/Sp(k-1,\mathbb{R}).$ Then
$$\text{\rm rank}_{\mathbb{R}}(\mathfrak{l}) \leq 2.$$ 
\label{p1}
\end{proposition}

\begin{proof}
Because $\mathop{\mathrm{rank}}\nolimits_{\text{a-hyp}}(\mathfrak{g}) = 1+ \mathop{\mathrm{rank}}\nolimits_{\text{a-hyp}}(\mathfrak{h})$ thus it follows from Table \ref{tab1} and Theorem \ref{twgl} that if $L$ is simple then $\text{rank}_{\mathbb{R}}(\mathfrak{l}) \leq 2.$ On the other hand if $L$ is semisimple then each (non-compact) simple part of $\mathfrak{l}$ adds at least $1$ to a-hyperbolic rank of $\mathfrak{l}.$ Thus we also have $\text{rank}_{\mathbb{R}}(\mathfrak{l}) \leq 2.$
\end{proof}

\textit{Proof of Corollary \ref{co1}.} Assume now that $L$ is reductive in $G.$ Since the Lie algebra $\mathfrak{l}$ is reductive therefore
$$\mathfrak{l} = \mathfrak{c}_{\mathfrak{l}} + [\mathfrak{l},\mathfrak{l}],$$
where $\mathfrak{c}_{\mathfrak{l}}$ denotes the center of $\mathfrak{l}.$ It follows from Corollary \ref{coko} that
\begin{equation}
\text{\rm rank}_{\mathbb{R}}(\mathfrak{l}) \leq k+1,
\label{eq2}
\end{equation}
and by Proposition \ref{p1} we have $\text{rank}_{\mathbb{R}}([\mathfrak{l},\mathfrak{l}]) \leq 2.$ Note that
$$d(G)-d(H)\geq k^{2} +2k +2.$$
We will show that if $L$ acts properly on $G/H$ and $k\geq 5$ then
\begin{equation}
d(L) < k^{2} + 2k +2.
\label{eq4}
\end{equation}
Let $[\mathfrak{l},\mathfrak{l}] = \mathfrak{k}_{0} + \mathfrak{p}_{0}$ be a Cartan decomposition. From (\ref{eq2})
\begin{equation}
d(L) \leq \text{\rm dim} (\mathfrak{c}_{\mathfrak{l}}) + \text{\rm dim} (\mathfrak{p}_{0}) \leq k+1 + \text{\rm dim} (\mathfrak{p}_{0}).
\label{eq7}
\end{equation}
Also, if $\text{rank}_{\mathbb{R}}([\mathfrak{l},\mathfrak{l}]) =2$ then it follows from Table \ref{tab1} that (the only) non-compact simple part of $[\mathfrak{l},\mathfrak{l}]$ is isomorphic to $\mathfrak{sl}(3,\mathbb{R}),$ $\mathfrak{su}^{\ast}(6),$ $\mathfrak{e}_{6}^{\text{IV}}$ or $\mathfrak{sl}(3,\mathbb{C})$ (treated as a simple real Lie algebra). In such case
\begin{equation}
\text{\rm dim} (\mathfrak{p}_{0}) < 27.
\label{eq5}
\end{equation}
Therefore assume that $\text{rank}_{\mathbb{R}} ([\mathfrak{l},\mathfrak{l}])=1$ and let $\mathfrak{s} \subset [\mathfrak{l},\mathfrak{l}]$ be (the only) simple part of a non-compact type. We have
\begin{equation}
\text{\rm rank}_{\mathbb{R}} (\mathfrak{s}) =1.
\label{eq3}
\end{equation}
It follows from Theorem \ref{twy} that $\mathfrak{s}$ is reductive in $\mathfrak{g}.$ Therefore $\mathfrak{s}$ admits a Cartan decomposition
$$\mathfrak{s} = \mathfrak{k}_{\mathfrak{s}} + \mathfrak{p}_{\mathfrak{s}}$$
compatible with $\mathfrak{g}= \mathfrak{k} + \mathfrak{p},$ that is $\mathfrak{k}_{\mathfrak{s}} \subset \mathfrak{k}.$ We also have $\text{dim}(\mathfrak{p}_{s})=\text{dim}(\mathfrak{p}_{0}).$ Since $\mathfrak{k} = \mathfrak{so}(2k+1)$ we obtain
$$\text{\rm rank} (\mathfrak{k}_{s}) \leq \text{\rm rank} (\mathfrak{k}) = k.$$
Using the above condition together with (\ref{eq3}) we can check (by a case-by-case study of simple Lie algebras) that
\begin{equation}
\text{\rm dim} (\mathfrak{p}_{\mathfrak{s}}) < 4k.
\label{eq6}
\end{equation}
Now (\ref{eq7}), (\ref{eq5}) and (\ref{eq6}) imply that
$$d(L) < 5k+1$$
for $k \geq 6,$ and $d(L)<33$ for $k=5.$ Thus we have showed (\ref{eq4}). The assertion follows from Theorem \ref{twkk}.

Maciej Boche\'nski 

              Department of Mathematics and Computer Science,
							
              University of Warmia and Mazury,
							
              S\l\/oneczna 54, 10-710, Olsztyn, Poland.
							
              email: mabo@matman.uwm.edu.pl         

        Marek Ogryzek 
				
              Department of Geodesy and Land Management,
							
              University of Warmia and Mazury,
							
              Prawoche\'nskiego 15, 10-720, Olsztyn, Poland. 
							
              email: marek.ogryzek@uwm.edu.pl


\begin{thebibliography}{99}

\bibitem{ben} Y. Benoist,  {\it Actions propres sur les espaces homog\`{e}nes r\'{e}ductifs},  Ann. of Math. 144 (1996),  315-347.

\bibitem{bt} M. Boche\'nski, A. Tralle, {\it Clifford-Klein forms and a-hyperbolic rank}, Int. Math. Res. Notices (2014),  DOI: 10.1093/imrn/rnu123.

\bibitem{bor} A. Borel,  {\it Compact Clifford-Klein forms of symmetric spaces},  Topology 2 (1963), 111-122.

\bibitem{kas} F. Kassel, {\it Proper actions on corank-one reductive homogeneous spaces}, J. Lie Theory 18 (2008), 961-978.

\bibitem{kako} F. Kassel, T. Kobayashi, {\it Poincar\'{e} series for non-Riemannian locally symmetric spaces}, arXiv:1209.4075.

\bibitem{kob2} T. Kobayashi, {\it Proper actions on a homogeneous space of reductive type}, Math. Ann. 285 (1989), 249-263.

\bibitem{kob3} T. Kobayashi, {\it Discontinuous groups acting on homogeneous spaces of reductive type}, Representation theory of Lie groups and Lie algebras (Fuji-Kawaguchiko, 1990), World
Sci. Publ., River Edge (1992),  59-75.

\bibitem{kobadm} T. Kobayashi, {\it A necessary condition for the existence of compact Clifford-Klein forms of homogeneous spaces of reductive type,} Duke Math. J. 67 (1992), 653-664.

\bibitem{kob4} T. Kobayashi, {\it On discontinuous groups acting on homogeneous spaces with noncompact isotropy subgroups}, J. Geom. Phys. 12 (1993), 133-144.

\bibitem{kob1} T. Kobayashi, {\it Criterion for proper actions on homogeneous spaces of reductive groups}, J. Lie Theory 6 (1996), 147-163.

\bibitem{ko} T. Kobayashi, T. Yoshino, {\it Compact Clifford-Klein forms of symmetric spaces revisited}, Pure Appl. Math. Quart. 1 (2005), 603-684.

\bibitem{kul} R. S. Kulkarni,  {\it Proper actions and pseudo-Riemannian space forms,}  Adv. Math. 40 (1981), 10-51.

\bibitem{mor} Y. Morita, {\it A topological necessary condition for the existence of compact Clifford-Klein forms}, arXiv:1310.7096, to appear in J. Differ. Geom.

\bibitem{ok} T. Okuda,  {\it Classification of semisimple symmetric spaces with $SL(2, \mathbb{R})$-proper actions,}  J. Differ. Geom. 94 (2013), 301-342.

\bibitem{ov2} A. L. Onishchik, E. B. Vinberg {\it Lie groups and algebraic groups,}  Springer (1990).

\bibitem{yo} K. Yosida, {\it A theorem concerning the semisimple Lie groups,} Tohoku Math. J. 44 (1938), 81-84.

\end{thebibliography}
\end{document}